\documentclass[11pt]{amsart}
\usepackage{amssymb}
\usepackage{amsthm}
\usepackage{amsfonts}
\newtheorem{lemma}{\bf Lemma}[section]

\newtheorem{proposition}[lemma]{\bf Proposition}

\newcommand{\Pz}{{\mathcal P}_2}

\begin{document}

\title{Hadwiger's conjecture: finite vs infinite graphs}

\author{Dominic van der Zypen}
\address{Federal office of social insurance, 
Effingerstrasse 20, CH-3003 Bern,
Switzerland}
\email{dominic.zypen@gmail.com}


\begin{abstract}
We study some versions of the statement of Hadwiger's
conjecture for finite as well
as infinite graphs.
\end{abstract}

\maketitle
\parindent = 0mm
\parskip = 2 mm
\section{Definitions}
In this note we are only concerned with simple
undirected graphs $G = (V, E)$ 
where $V$ is a set and $E \subseteq \Pz(V)$ where
$$\Pz(V) = \big\{ \{x,y\} : x,y \in V\textrm{ and } x\neq y\big\}.$$ 
We denote the vertex set of a graph $G$ by $V(G)$ and the edge set by 
$E(G)$. Moreover, for any cardinal $\alpha$ we denote the complete
graph on $\alpha$ points by $K_\alpha$. 

For any graph $G$, disjoint subsets $S, T \subseteq V(G)$ 
are said to be {\em connected
to each other} if there are $s \in S, t\in T$ with $\{s,t\}\in E(G)$.

Given a collection ${\mathcal D}$ of pairwise disjoint, nonempty, connected
subsets of $V$, we associated with ${\mathcal D}$ a graph $G({\mathcal D})$
with vertex set ${\mathcal D}$ and 
$$E(G({\mathcal D})) = \big\{\{d,e\}: d\neq e
\in {\mathcal D} \text{ and } d, e 
\textrm{ are connected to each other}\big\}.$$ We say that a graph $M$
is a {\em minor} of a graph $G$ if there is a 
collection ${\mathcal D}$ of pairwise disjoint, nonempty, connected
subsets of $V$ and an injective graph homomorphism 
$f:M\to G({\mathcal D})$.

This implies that $K_\alpha$ is a {\em minor} of a graph $G$ if and only if 
there is 
a collection $\{S_\beta: \beta \in \alpha\}$ of nonempty, connected and 
pairwise disjoint subsets of $V(G)$ such that for all $\beta,\gamma \in \alpha$
with $\beta \neq \gamma$ the sets $S_\beta$ and $S_\gamma$ are connected
to each other. 

\section{Different statements of Hadwiger's conjecture}
The statement of Hadwiger's conjecture that 
is usually found in the literature is the following:
\begin{quote}
\textbf{(H)}: If $G$ is a simple undirected graph and $\lambda= \chi(G)$
then the complete graph $K_\lambda$ is a minor of $G$.
\end{quote}
The next version of Hadwiger's statement has a bit of a different
flavor, and we will compare it to (H) in the finite
and infinite contexts in the following sections.
\begin{quote}
\textbf{(ModH)}: For every graph $G$ there is a minor $M$ of $G$ such that
\begin{enumerate}
\item $M\not \cong G$, and
\item $\chi(M) = \chi(G)$.
\end{enumerate}
\end{quote}
There is a version of (ModH) that has appears to be similar,
but we will see later that it is worthwhile to look
at the statement separately.
\begin{quote}
\textbf{(HomH)}: For every graph $G$ there is a minor $M$ of $G$ such that
\begin{enumerate}
\item $M\not \cong G$, and
\item there is a graph homomorphism $f:G\to M$.
\end{enumerate}
\end{quote}
Last, the following weaker version of this was studied in \cite{me}:
\begin{quote}
\textbf{(WeakH)}: Whenever $\lambda$ is a cardinal 
such that there is no graph homomorphism 
$c: G\to K_\lambda$ then $K_\lambda$ is a minor of $G$.
\end{quote}

\section{The finite case}
Overview:
\begin{itemize}
\item (H) is a long-standing open problem.
\item (ModH) is equivalent to (H) for finite graphs 
(see proposition \ref{fin}).
\item (HomH) is also equivalent to (H) for finite graphs.
\item (WeakH) is implied by (H).
\end{itemize}
\begin{proposition}\label{fin}
For finite graphs $G$, the statements (H) and (ModH) are equivalent.
\end{proposition}
\begin{proof} Given a finite non-complete graph $G=(V,E)$, the statement
(H) implies that $K=K_{\chi(G)}$ is a minor of $G$. Since $K$ is 
complete, but not $G$, they are not isomorphic, so (ModH) holds.

For the other implication, take any finite graph $G$ and let $n=\chi(G)$.
Use (ModH) to get a proper minor $M_1$ such that $\chi(M_1)=n$.
If $M_1$ is complete, we have proved (H), otherwise use (ModH) again
to find a proper minor $M_2$ of $M_1$ with $\chi(M_2)=n$, and so on.
Since $G$ is finite, this procedure is bound to end at some $M_k$
for some $k\in\mathbb{N}$, which implies that $M_k$ is complete 
and has $n$ points.
\end{proof}

It is easy to modify Proposition \ref{fin} to see that
in the finite case, (H) and (HomH) are equivalent.

In the finite setting, the statement (WeakH) amounts to saying
that if $\chi(G)=t>0$ then $K_{t-1}$ is a minor of $G$. This
is weaker than (H); whether it is strictly weaker is an open
question (see section \ref{open}).
\section{The infinite case}
\subsection{Infnite chromatic number}

Overview:
\begin{itemize}
\item (H) is \textbf{false}: Let $G$ be the disjoint union of all $K_n, 
n\in\mathbb{N}$. Then $\chi(G)=\omega$, but $K_\omega$ is
not a minor of $G$.
\item (ModH) is \textbf{true}, see proposition \ref{propinf}.
\item (HomH) is open.
\item (WeakH) is \textbf{true}, see \cite{me}.
\end{itemize}
So that is why we sepatately introduced (HomH) in addition to
(ModH): they might be different for graphs with infinite
chromatic number.
\begin{proposition}\label{propinf}
For graphs with infinite chromatic number, (ModH) is true.
\end{proposition}
\begin{proof}
Let $I$ be the set of isolated vertices of $G$.

\textit{Case 1.} $I\neq \emptyset$. We set $M=G\setminus I$.
It is easy to see that $M\not\cong G$ as $M$ contains no isolated
points. Since $\chi(G)\geq \aleph_0$ we have $\chi(M)=\chi(G)$.

\textit{Case 2.} $I=\emptyset$. Fix $v_0\in V(G)$. Let
$M=(V(G), E)$ where $$E = \{e\in E(G): v_0\notin e\},$$ that
is we remove all edges connecting $v_0$ to some other vertex
in $V(G)$. Since $M$ has $v_0$ as an isolated point, but $G$ 
has no isolated points, we have $M\not\cong G$, and it
is easy to verify that $\chi(M)=\chi(G)$.
\end{proof}

\subsection{Finite chromatic number}

For infinite graphs with finite chromatic number
we get the following results:
\begin{itemize}
\item It is not known whether (H) and (WeakH) are true;
\item (ModH) is \textbf{true}: the theorem of De Bruijn and Erd\H{o}s
\cite{DeBE} implies that if $G$ is infinite with finite chromatic 
number, there is
a finite subgraph $M$ of $G$ with $\chi(M) = \chi(G)$.
\item (HomH) is true for the same reason (note that a coloring
is always a graph homomorphism to a complete graph).
\end{itemize}

\section{Open questions}\label{open}
{\it Question 1.} Does the weak Hadwiger conjecture
(WeakH) hold for finite graphs?

(WeakH) might be as elusive has (H) has been so far;
so here is a different problem:

{\it Question 2.} When we restrict ourselves to finite graphs,
does the weak Hadwiger conjecture (WeakH) 
imply the statement of the Hadwiger conjecture?

The next question is a stronger version of (ModH) and focuses
on finite graphs.

{\it Question 3.} Suppose that $G$ is a finite, connected graph
such that whenever you contract 1 edge or 2 edges, the
chromatic number decreases. Does this imply $G$ is complete?

Finally we turn to infinite graphs:

{\it Question 4.} Does (WeakH) hold for infinite graphs 
with finite chromatic number?

{\it Question 5.} Does (HomH) hold for graphs with 
infinite chromatic number?

\section{Acknowledgement}
I would like to thank user \texttt{@bof} from \texttt{mathoverflow.net}
for his argument used in proposition \ref{propinf} \cite{bof}.

\end{document}